\documentclass[11pt]{article}

\pdfoutput=1

\usepackage{geometry}
\usepackage{
	amsmath,
	amsthm,
	thmtools,
	thm-restate
}

\usepackage{
	amsfonts,
	amssymb,
	mathtools,
	mathpartir,
	color,
	mdframed,
	mathrsfs
}

\usepackage{epsdice}

\usepackage[labelfont=bf, font=small]{caption}

\usepackage{authblk}

\usepackage{url}
\usepackage{hyperref}
\hypersetup{
	colorlinks=true,
	linkcolor=blue,
	citecolor=blue,
    urlcolor=blue
}

\usepackage[capitalize]{cleveref}

\usepackage{tikz, tikz-cd}
\usetikzlibrary{
	automata,
	arrows,
	positioning,
	commutative-diagrams
}
\usepackage{prftree}

\usepackage{enumitem}

\usepackage[
backend=biber,
style=ieee,
url=false
]{biblatex}
\title{Presenting with Quantitative Inequational Theories}
\author{Todd Schmid\footnote{This work was partially supported by ERC grant Autoprobe (grant agreement 101002697).}}
\affil{University College London}
\date{}

\begin{document}

\theoremstyle{plain}
\newtheorem{theorem}{Theorem}
\newtheorem{lemma}[theorem]{Lemma}
\newtheorem{corollary}[theorem]{Corollary}
\newtheorem{proposition}[theorem]{Proposition}

\theoremstyle{definition}
\newtheorem{remark}[theorem]{Remark}
\newtheorem{definition}[theorem]{Definition}
\newtheorem{conjecture}[theorem]{Conjecture}
\newtheorem{example}[theorem]{Example}
\newtheorem{assumption}{Assumption}
\newtheorem{question}{Question}

\tikzset{every state/.style={minimum size=0pt}}
\tikzset{every node/.style={scale=1}}

\makeatletter
\newcommand*{\da@rightarrow}{\mathchar"0\hexnumber@\symAMSa 4B }
\newcommand*{\da@leftarrow}{\mathchar"0\hexnumber@\symAMSa 4C }
\newcommand*{\xdashrightarrow}[2][]{%
  \mathrel{%
    \mathpalette{\da@xarrow{#1}{#2}{}\da@rightarrow{\,}{}}{}%
  }%
}
\newcommand*{\da@xarrow}[7]{%
  \sbox0{$\ifx#7\scriptstyle\scriptscriptstyle\else\scriptstyle\fi#5#1#6\m@th$}%
  \sbox2{$\ifx#7\scriptstyle\scriptscriptstyle\else\scriptstyle\fi#5#2#6\m@th$}%
  \sbox4{$#7\dabar@\m@th$}%
  \dimen@=\wd0 %
  \ifdim\wd2 >\dimen@
    \dimen@=\wd2 %
  \fi
  \count@=2 %
  \def\da@bars{\dabar@\dabar@}%
  \@whiledim\count@\wd4<\dimen@\do{%
    \advance\count@\@ne
    \expandafter\def\expandafter\da@bars\expandafter{%
      \da@bars
      \dabar@ 
    }%
  }%
  \mathrel{#3}%
  \mathrel{%
    \mathop{\da@bars}\limits
    \ifx\\#1\\%
    \else
      _{\copy0}%
    \fi
    \ifx\\#2\\%
    \else
      ^{\copy2}%
    \fi
  }%
  \mathrel{#4}%
}
\makeatother

\newcommand 	\sle 		{\sqsubseteq} 		
\newcommand 	\sge 		{\sqsupseteq}		

\newcommand 	\At  		{A} 		
\newcommand 	\Eq 		{\mathscr{E}}		
\newcommand 	\Ie 		{\mathscr{I}} 		
\newcommand 	\Exp 		{\mathtt{Exp}} 		
\newcommand 	\Expm 		{{\Exp}/{\equiv}}	
\newcommand 	\Expi 		{{\Exp/{\sle}}}		
\newcommand 	\SExp 		{\mathtt{SExp}} 	
\newcommand 	\acro[1]	{\(\mathsf{#1}\)} 				

\newcommand 	\Sets		{\mathsf{Sets}}				
\newcommand 	\Met 		{\mathbs{Met}} 				
\newcommand 	\Pos		{\mathsf{Pos}} 				
\newcommand  	\dCPO 		{\mathsf{dCPO}}				
\newcommand 	\dcpo  		{{\textsf{dcpo}}}
\newcommand 	\Cat 		{\mathsf{C}}
\newcommand 	\Alg 		{\operatorname{Alg}}
\newcommand 	\Coalg 		{\operatorname{Coalg}}  	

\newcommand   \PMod[1]   {\mathbf {#1}\mathsf{Mod}}

\newcommand 	\Id 		{\text{Id}}					
\renewcommand 	\P			{\mathcal{P}_{\omega}}		
\newcommand 	\M			{\mathcal{M}_{\omega}}	
\newcommand 	\N 			{\mathbb N}				
\newcommand 	\D 			{\mathcal{S}}	
\newcommand 	\C			{\mathcal{C}}			
\renewcommand 	\O			{\mathcal{O}}			
\newcommand 	\A			{\mathcal{A}_\omega}			
\newcommand 	\PA 		{\mathsf{PA}}

\newcommand 	\E 			{{\mathtt{Prc}}}
\newcommand 	\X 			{{\mathbb X}}
\newcommand 	\Y 			{{\mathbb Y}}

\newcommand 	\id 			{\operatorname{id}}			
\newcommand 	\supp 			{\operatorname{supp}}		
\newcommand 	\conv 			{\operatorname{conv}}	
\newcommand 	\fv 			{\operatorname{fv}}	
\newcommand 	\bv 			{\operatorname{bv}}		
\newcommand 	\downset 		{\operatorname{\downarrow}}
\newcommand 	\Pfp 			{\operatorname{Pfp}}
\newcommand 	\lfp 			{\operatorname{lfp}}
\newcommand 	\fp 			{\operatorname{fp}}
\newcommand 	\disc  			{\operatorname{Dis}}
\newcommand 	\forg[1] 		{\left|#1\right|}

\newcommand 	\sub 			{\mathbin{/}}
\newcommand 	\gdsub			{\mathbin{/\!/}}
\newcommand 	\unsub 			{\raisebox{-0.3em}{
	\begin{tikzpicture}
		\draw[thick, dotted] (0,-0.2) -- (1ex, 1.2ex);
	\end{tikzpicture}
}}

\newcommand 	\upper 		{\operatorname{\uparrow}}
\newcommand 	\lhdeq  	{\unlhd}

\newcommand 	\qm 		{\mathbin{?}}

\newcommand 	\tr[1] 		{\mathrel{\raisebox{-0.1em}{{\footnotesize\(\xrightarrow{#1}\)}}}}
\newcommand 	\out[1] 	{\mathrel{\raisebox{-0.1em}{{\footnotesize\(\xRightarrow{#1}\)}}}}
\newcommand 	\trd[1] 	{\mathrel{\raisebox{-0.1em}{{\footnotesize\(\xdashrightarrow{#1}{}\)}}}}

\newcommand 	\Act 		{\mathtt{Act}} 		
\newcommand 	\Var 		{\mathtt{Var}}		
\newcommand 	\code[1] 	{\mathtt{#1}}

\newcommand 	\sem[1] 	{\lceil \!\! \lfloor #1 \rceil \!\! \rfloor}
\newcommand 	\asem[1] 	{\lceil \!\! \lceil #1 \rceil \!\! \rceil}
\newcommand 	\ev 		{\operatorname{ev}}
\newcommand 	\beh   		{{!}}
\newcommand 	\gd			{\mathsf{gd}} 				
\newcommand 	\skiptt  	{\code{1}} 			
\newcommand 	\failtt 	{\code{0}}
\newcommand 	\unit		{\underline{\code{u}}}
\newcommand 	\eff  		{\code{f}}

\newcommand 	\defn[1] 	{\emph{#1}}
\newcommand 	\note[1]  	{{\color{red} #1}}
	
	\maketitle
	
    \begin{abstract}
        It came to the attention of myself and the coauthors of~\cite{schmidrozowskisilvarot2022processes} that a number of process calculi can be obtained by algebraically presenting the branching structure of the transition systems they specify. 
        For example, labelled transition systems branch into finite sets of transitions representable by terms in the free semilattice generated by the transitions and probabilistic labelled transition systems branch into finitely supported probability distributions representable by terms in the free convex algebra generated by the transitions~\cite{bonchi_silva_sokolova:2017,flood1981semiconvex,schmidrozowskisilvarot2022processes}.
    
        Since the mentioned theories are equational, their presentations are given in terms of monads on the category of sets, but restricting branching structures to set-valued monads has a number of undesirable limitations.
        In~\cite{schmid2022ordered}, I show how to extend the framework of \cite{schmidrozowskisilvarot2022processes} to ordered branching structures given by monads on the category of partially ordered sets.
        Unfortunately, ordered analogues of a few important examples of monad presentations seem to be missing from the literature, despite their relevance to process theory beyond labelled transition systems. 
        My goal in this article is to initiate a line of research that addresses this gap in a way that lends itself well to studying behavioural inequalities of weighted automata and related models.

        In this article, I discuss examples of monads on the category of partially ordered sets coming from quantitative theories, free modules over ordered semirings, and give sufficient conditions for one of these to lift a monad on the category of sets. 
        I also discuss monad presentations in general.
        In the last section, I give descriptions of ordered semirings that are useful for specifying unguarded recursive calls. 
        Applications include ordered probability theory and ordered semilattices. 
    \end{abstract}

	\section{Monad presentations}
    Fix an endofunctor \(S : \Cat \to \Cat\) on a category \(\Cat\). 
    An (\(S\)-)\defn{algebra} is a pair \((X, \alpha)\) consisting of an object \(X\) of \(\Cat\) and an arrow \(\alpha : SX \to X\). 
    A \defn{homomorphism} \(h : (X, \alpha) \to (Y, \beta)\) is an arrow \(h : X \to Y\) in \(\Cat\) such that \(S(h) \circ \alpha = \beta \circ h\).
    \(\Alg_\Cat(S)\) denotes the category of \(S\)-algebras and homomorphisms.

    Intuitively, a monad is a free-algebra construction that takes some desirable algebraic properties as input and outputs the initial algebra satisfying those properties.
    In the most general picture, a property is simply a full subcategory of \(\Alg(S)\). 
    The following notion of monad presentation is inspired by~\cite{bonchiSV2019tracesfor} (see also~\cite[IV.\S8]{maclane1971working}).

    \begin{definition}
        Let \(\mathsf T\) be a full subcategory of \(\Alg_\Cat(S)\).
        A \defn{\(\mathsf T\)-presented monad} is a triple \((M, \eta, \rho)\) consisting of the following ingredients: (i) An endofunctor \(M : \Cat \to \Cat\), (ii) a natural transformation \(\eta : \Id \Rightarrow M\), and (iii) a natural transformation \(\rho : SM \Rightarrow M\) such that 
            \begin{enumerate}
                \item[(a)] \((MX, \rho_X) \in \mathsf T\) for any object \(X\) of \(\Cat\), and
                \item[(b)] for any \(S\)-algebra \((Y, \beta)\) and any arrow \(f : X \to Y\) in \(\Cat\), there is a unique \(S\)-algebra homomorphism \(f^\beta : (MX, \rho_X) \to (Y, \beta)\) such that \(f^\beta \circ \eta = f\).
            \end{enumerate}
    \end{definition}

    This notion of monad is not entirely standard. 
    To see its connection to the usual notion, we turn to the following two results. 

    \begin{restatable}{theorem}{wegetmonads}
        Every \(\mathsf T\)-presented monad \((M, \eta, \rho)\) corresponds to a unique monad \((M, \eta, \mu)\) on \(\Cat\). 
    \end{restatable}

    \begin{proof}
        Consider the identity arrow \(\id_{MX} : MX \to MX\). 
        We define the natural transformation \(\mu : MM \Rightarrow M\) to have the components \(\mu_X = \id_{MX}^{\rho_{MX}}\), where \(\id_{MX}^{\rho_{MX}}\) is the unique \(S\)-algebra homomorphism \((MMX, \rho_{MX}) \to (MX, \rho_X)\) such that \(\id_{MX}^{\rho_{MX}} \circ \eta = \id_{MX}\).
        It is now routine to check that \((M, \eta, \mu)\) is a monad. 
        
        Given an arbitrary monad \((M, \eta, \mu)\), one simply needs to verify that \(\mu_{X} : MMX \to MX\) is an \(S\)-algebra homomorphism \((MMX, \rho_{MX}) \to (MX, \rho_X)\). 
    \end{proof}

    \begin{theorem}\label{thm:T-pres equiv to EM-algs}
        The functor \(\Phi : (X, \alpha) \mapsto (X, \id_X^\alpha)\) is a functor between \(\mathsf T\) and the category of Eillenberg-Moore algebras \(\operatorname{EM}(M, \eta, \mu)\), where \(\mu_{X} = \id_{MX}^{\rho_{MX}}\).
        If \(\mathsf T\) is closed under split epis (i.e. if \((X, \alpha) \in \mathsf T\) and \((\exists m,i)~m : (X, \alpha) \to (Y, \beta)\) and \(m \circ i = \id_Y\), then \((Y, \beta) \in \mathsf T\)), then \(\Phi\) is an isomorphism of categories. 
    \end{theorem}
    
    \begin{proof}
        By assumption, \((X, \beta) \in \mathsf T\), so the universal property of \(\rho\) applies to it.
        We start by verifying that \((X, \id_X^\beta)\) is indeed an Eillenberg-Moore algebra.
        By naturality of \(\eta\) and \(\rho\), the following diagram commutes. 
        \[\begin{tikzcd}
            SMMX \ar[r] \ar[d, "\rho_M"] 
            & SMX \ar[r] \ar[d, "\rho"] 
            & SX \ar[d, "\beta"] \\
            MMX \ar[r, "M(\id_X^\beta)"] 
            & MX \ar[r, "\id_X^\beta"] 
            & X \\
            MX \ar[u, "\eta"] \ar[r, "\id_X^\beta"] 
            & X \ar[u, "\eta"] \ar[ur, "\id_X"']
        \end{tikzcd} \]
         By the universal property of \(\rho\), \(\id_X^\beta \circ M(\id_X^\beta)\) is the unique \(S\)-algebra homomorphism satisfying \[(\id_X^\beta \circ M(\id_X^\beta)) \circ \eta = \id_X^\beta\]
         On the other hand, the following commutes because \(\mu \circ \eta_M = \id_M\). 
         \[\begin{tikzcd}
             SMMX \ar[r] \ar[d, "\rho_M"] 
             & SMX \ar[r] \ar[d, "\rho"] 
             & SX \ar[d, "\beta"] \\
             MMX \ar[r, "\mu"] 
             & MX \ar[r, "\id_X^\beta"] 
             & X \\
             MX \ar[u, "\eta_M"] \ar[ur, "\id_{MX}"']
         \end{tikzcd} \]
        This means that \(\id_X^\beta \circ \mu_M\) is also an \(S\)-algebra homomorphism satisfying \((\id_X^\beta \circ \mu_M) \circ \eta = \id_X^\beta\).
        It follows that \(\id_X^\beta \circ \mu_M = \id_X^\beta \circ M(\id_X^\beta)\).
        Since \(\id_X^\beta \circ \eta = \id_X\) by definition, \((X, \id_X^\beta)\) is an EM-algebra for \((M, \eta, \mu)\).
        This establishes that \(\Phi\) is indeed a functor into \(\operatorname{EM}(M, \eta, \mu)\). 
        
        Now we construct the inverse of \(\Phi\).
        Let \(\Theta(X, \gamma) = (X, \gamma \circ \rho \circ S(\eta))\). 	
        We will show that \((X, \gamma \circ \rho \circ S(\eta)) \in \mathsf T\) by exhibiting a split epi \(MX \to X\). 
        In fact, the split epi is \(\gamma\), since \(\gamma \circ \eta = \id_X\), so we just need to show that \(\gamma\) is an \(S\)-algebra homomorphism.
        This can be seen from the following diagram
        \[\begin{tikzcd}
            SMX \ar[d, "S(\gamma)"] \ar[r, "S(\eta_M)"] 
            & SMMX \ar[d, "SM(\gamma)"] \ar[r, "\rho_M"]
            & MMX \ar[r, "\mu"] \ar[d, "M(\gamma)"]
            & MX \ar[d, "\gamma"]\\
            SX \ar[r, "S(\eta)"] 
            & SMX \ar[r, "\rho"] 
            & MX \ar[r, "\gamma"] 
            & X	\end{tikzcd}\]
        This diagram commutes by naturality of \(\mu\), \(\rho\), \(S(\eta)\), and the fact that \((X,\gamma)\) is am EM-algebra. 
        We furthermore know that 
        \begin{align*}
            \mu \circ \rho_M \circ S(\eta) 
            &= \mu \circ \eta_M \circ \rho_M \tag{naturality} \\
            &= \id_M \circ \rho_M \tag{def. of monad} \\
            &= \rho
        \end{align*}
        Hence, \(\Theta : \operatorname{EM}(M,\eta, \mu) \to \mathsf T\).
        Furthermore, \(\gamma\) is the unique \(S\)-algebra homomorphism \((MX, \rho) \to \Theta(X, \gamma)\) such that \(\gamma \circ \eta = \id_X\), so in fact \(\Phi\Theta(X, \gamma) = (X, \gamma)\). 
        It therefore suffices to see that \(\Theta\Phi(X, \beta) = (X, \beta)\).
        This is rather easy, fortunately:
        \begin{align*}
            \id_X^\beta \circ \rho \circ S(\eta) 
            = \beta \circ S(\id_X^\beta) \circ S(\eta) 
            = \beta \circ S(\id_X^\beta \circ \eta) 
            = \beta \circ S(\id_X)  
            = \beta \tag*{\qedhere}
        \end{align*}
    \end{proof}

    \begin{remark}
        The assumption that \(\mathsf T\) is closed under split epis is not very strict. 
        For example, every variety is closed under split epis, because every split epi is automatically regular~\cite{adamek1990abstract,adamekP03varieties}. 
    \end{remark}
    
    \section{Quantitative theories}
    Discrete processes that branch with observable quantities are popular in the automata theory and process algebra literature~\cite{stark1999complete,andova1999process,bonchi_silva_sokolova:2017,bonchibonsangueborealeruttensilva2012weighted}.
    In the cited works, formal calculi are used for specifying and reasoning about behaviours of such processes.
    Algebraically reasoning about behaviour inevitably relies on an algebraic characterisation of the branching structures of processes.

    Quantitative branching is often depicted by decorating transitions with numbers, vectors, or other quantities. 
    Generally, quantities appear as elements of a \emph{semiring} \(\mathbf P = (P, 0, 1, +, \cdot)\), a set \(P\) equipped with constants \(0,1\) and binary operations \(+,\cdot\) such that \((P, 0, +)\) is a commutative monoid, \((P, 1, \cdot)\) is a monoid, \(0r = 0\), \(r(s + t) = rs + rt\), and \((s + t)r = sr + tr\) for all \(r,s,t \in P\)~\cite{golan2013semirings}.

    \begin{definition}
        A \emph{\(\mathbf P\)-module} is a monoid \((X, 0, +)\) equipped with an action \(P \times X \to X\), written \((r, x) \mapsto rx\), such that \(0x = 0\), \(1x = x\), \((rs)x = r(sx)\), and \((r + s)x = rx + sx\) for \(x \in X\) and \(r,s \in P\).
    \end{definition}
    
    Abstractly, \(\mathbf P\)-modules are algebras for \(S = P \times \Id + \{+\}\times \Id^2\) that satisfy a handful of equations, where \(\Id\) is the identity on \(\Sets\). 
    They form a full subcategory \(\PMod{P}\) of \(\Alg_\Sets(S)\).
    The free \(\mathbf P\)-module on \(X\) consists of finitely supported \(\theta : X \to P\), where the \emph{support} of \(\theta\) is \(\supp(\theta) = \{x \in X \mid \theta(x) \neq 0\}\).

    \begin{definition}
        The \emph{free \(\mathbf P\)-module} functor \(\O_{\mathbf P} : \Sets \to \Sets\) is given by
        \[
            \O_{\mathbf P} X = {\{\theta : X \to P \mid \supp(\theta)\text{ is finite}\}}
            \qquad \qquad
            \O_{\mathbf P}(h)(\theta)(y) = h^\bullet\theta(y) := \sum_{h(x) = y} \theta(x)
        \] 
        for any set \(X\) and any \(h : X \to Y\).
        The \(\PMod{P}\)-presented monad is the triple \((\O_{\mathbf P}, \delta, \rho)\), defined
        \[\delta_X(x) = \delta_x := \lambda y.\begin{cases}
            1 & x = y \\
            0 & x \neq y
        \end{cases}
        \qquad 
        \rho(0) = 0
        \qquad 
        \rho(r, \theta) = r \theta
        \qquad 
        \rho(+, \theta_1, \theta_2) = \theta_1 + \theta_2
        \]
        where \(0(x) = 0\), and \(r\theta\) and \(\theta_1 + \theta_2\) are evaluated pointwise.
    \end{definition}

    \begin{example}\label{eg:semilattices}
        Consider \(\mathbf 2 = (\{0,1\}, 0, 1, \max, \min)\).
        \(\PMod{2}\) is equivalent to the category of join semilattices with bottom.
        In particular, \(\O_{\mathbf 2}\) is naturally isomorphic to the finitary powerset functor \(\P\).
    \end{example}

    \begin{example}\label{eg:reals}
        Consider \(\mathbf R^+ = (\mathbb R_{\ge 0}, 0, 1, +, \times)\).
        \(\PMod{R^+}\) is equivalent to the category of positive cones of finite dimensional real-valued metric spaces and linear maps with nonnegative entries.
    \end{example}

    \section{Ordered quantitative theories} 
    A notable example of the power of algebraic presentation for the purposes of reasoning about behaviour is Stark and Smolka's algebra of probabilistic actions~\cite{stark1999complete}, which is used to study processes that branch probabilistically, with weights from the semiring \(\mathbf R^+\).  
    An important feature of their calculus is its interpretation of recursive specifications as least fixed-points. 

    Recently~\cite{schmid2022ordered}, I made that the observation that the existence of these least fixed-points is a property of the \emph{inequational theory} obtained from the theory of \(\mathbf R^+\)-modules. 
    Where \(\Pos\) is the category of partially ordered sets (posets) and monotone maps, an inequational theory is a set of inequalities that describes a category of \(S\)-algebras for a functor \(S : \Pos \to \Pos\), or \emph{ordered algebras}.

    \begin{definition}
        An \emph{ordered semiring} \(\mathbf P = (P, \le, 0, 1, +, \cdot)\) is a semiring \((P, 0, 1, +, \cdot)\) where \((P, \le)\) is a poset with \(0\) as its bottom element\footnote{These are sometimes called \emph{positive} ordered semirings in the literature~\cite{golan2013semirings}.}, and \(+\) and \(\cdot\) are monotone.
        An \emph{ordered \(\mathbf P\)-module} is a sturcture \((X, \le, 0, +)\) consisting of a \(\mathbf P\)-module \((X, 0,  +)\) such that \(r \le s\) and \(x \le y\) implies \(rx + z \le sy + z\) for any \(r,s \in P\) and \(x,y,z \in X\).
    \end{definition}

    Abstractly, an ordered \(\mathbf P\)-module is an algebra for \(S = (P, \le) \times \Id + (\{+\}, =)\times \Id^2\) satisfying a set of inequations, where \(\Id\) is the identity functor on \(\Pos\) now. 
    We also write \(\PMod{\mathbf P}\) for the full subcategory of \(\Alg_\Pos(S)\) consisting of ordered \(\mathbf P\)-modules. 

    If the reader is anything like myself, they might expect me to say that \(\O_{\mathbf P}\) constructs the free ordered \(\mathbf P\)-module on a poset \((X, \le)\).
    This is not exactly the case.
    For example, turn the semiring \(\mathbf 2\) from \cref{eg:semilattices} into an ordered semiring by taking \(0 < 1\). 
    Then an ordered \(\mathbf 2\)-module consists of the same data as a semilattice with bottom, in the sense of order theory~\cite{birkhoff1940lattice}. 
    The free ordered semilattice on a poset \((X, \le)\) is carried by the finitely generated downwards-closed subsets of \((X, \le)\), which is not even the same size as \(\P X\) in general!
    
    In many cases, quotienting the set \(\O_{\mathbf P} X\) by a certain preorder suffices.
    Call a subset \(U \subseteq X\) \emph{upwards closed} if \(x \in U\) and \(x \le y\) implies \(y \in U\). 
    
    \begin{definition}
        The \emph{heavier-higher order} is the preorder \(\sqsubseteq\) on \(\O_{\mathbf P} X\) defined so that \(\theta_1 \sqsubseteq \theta_2\) if and only if for any upwards-closed \(U \subseteq X\), \(\sum_{x \in U} \theta_1 (x) \le \sum_{x \in U} \theta_2(x)\) (see for e.g.~\cite{shaked1994stochastic}).
    \end{definition}
    
    We write \(\theta_1 \equiv \theta_2\) if \(\theta_1 \sqsubseteq \theta_2 \sqsubseteq \theta_1\), \([\theta_1] = \{\theta_2 \mid \theta_1 \equiv \theta_2\}\), and \(\bar\O_{\mathbf P} (X, \le) = (\O_{\mathbf P} X/{\equiv}, \sqsubseteq)\).    
    The next lemma tells us that \(\bar\O_{\mathbf P}\) is an endofunctor on \(\Pos\). 
    
    \begin{lemma}
        For any monotone map \(f : (X, \le) \to (Y, \le)\), the map \(f^\bullet : \mathcal O_\omega X \to \mathcal O_\omega Y\) is monotone w.r.t. the heavier-higher order. 
    \end{lemma}
    
    \begin{proof}
        Let \(U \subseteq Y\) be an upwards closed subset of \((Y, \le)\).
        Then \(f^{-1}(U)\) is an upwards closed subset of \((X, \le)\).
        We have
        \begin{align*}
            f^\bullet\theta_1(U) &= \sum_{u \in U} f^\bullet\theta_1(u) 
            = \sum_{u \in U} \sum_{f(x) = u} \theta_1(x) 
            = \sum_{f(x) \in U} \theta_1(x) 
            \le \sum_{f(x) \in U} \theta_2(x) 
            = f^\bullet\theta_2(U) 
        \end{align*}
    \end{proof}

    Next we indicate the cases where free ordered modules are constructed by the functor \(\bar\O_{\mathbf P}\).
    Say that a semiring is \emph{cancellative} if \((\forall r,s,t)~r + s \le r + t\) implies \(s \le t\), \emph{difference ordered} if \(r \le s\) implies \((\exists t)~t + t = s\), and \emph{idempotent} if \(r + r = r\).

    \begin{theorem}\label{thm:main}
        If \(\mathbf P\) is either cancellative difference-ordered or idempotent, then \({(\bar\O_{\mathbf P}, [\delta], [\rho])}\) is the \(\PMod{\mathbf P}\)-presented monad. 
    \end{theorem}

    This follows from the next three facts.
    \begin{description}
        \item[Fact (1)] The transformation \(\delta\) is monotone with respect to \(\sqsubseteq\).
        \item[Fact (2)] \((\bar\O_{\mathbf P} (X, \le), \rho)\) is an ordered \(\mathbf P\)-module.
        This means that for \(\theta_1,\theta_2, \varphi \in \mathcal O_\omega X\) and \(p, q \in P\),
        \begin{itemize}
            \item[(i)] if \(\theta_1 \sqsubseteq \theta_2\), then \(\theta_1 + \varphi \sqsubseteq \theta_2 +\varphi\), and
            \item[(ii)] if \(p \le q\) (in \(\mathbb P\)), then \(p  \varphi \sqsubseteq q  \varphi\).
        \end{itemize} 
        \item[Fact (3)] For any ordered \(\mathbf P\)-module \((Y, \le, \beta)\) and any monotone \(f : (X, \le) \to (Y, \le)\), the homomorphism \(f^\beta : (\mathcal O_\omega X, \rho) \to (Y, \beta)\) defined
        \[
            f^\beta(0) = 0
            \qquad 
            f^\beta(\delta_x) = f(x) 
            \qquad 
            f^\beta(r \theta) = r f^\beta(\theta) 
            \qquad 
            f^\beta(\theta_1 + \theta_2) = f^\beta(\theta_1) + f^\beta(\theta_2)
        \]
        is monotone with respect to the heavier-higher order.
    \end{description}
    We are only going to argue for Fact (3), as it is the only one with a nontrivial proof (additionally, Facts (1) and (2) do not require the cancellative, difference ordered, or idempotent assumptions).
    The argument uses the following lemma, whose proof we omit. 

    \begin{lemma}
        Let \(\mathbf P\) be an ordered semiring. 
        \begin{enumerate}
            \item If \(\mathbf P\) is cancellative and difference ordered, then \(r \le s\) implies there exists a unique \(t \in P\) such that \(r + t = s\).
            (We write the witness as \(t = s - r\).)
            \item If \(\mathbf P\) is idempotent, then in every ordered \(\mathbf P\)-module \((X, \le, +)\), \(x \le y\) if and only if \(x + y = y\).
        \end{enumerate}
    \end{lemma}

    \begin{proof}[Proof of Fact (3).]
        Let \(f : (X, \le) \to (Y, \le)\) be a monotone map into an ordered \(\mathbf P\)-module \((Y, \le, 0,  +)\), and let \(\theta_1 \sqsubseteq \theta_2\). 
        We proceed by induction on the support of \(\theta_1\).

        Assume \(\mathbf P\) is cancellative difference-ordered. 
        Let \(z\) be any maximal element of \(\supp(\theta_1)\), so that \(\theta_1(z) = \theta_1(\upper z)\).
        Define 
        \[
            \theta_i'(x) = \begin{cases}
                \theta_i(z) - \theta_1(z) &x = z \\
                \theta_i(x) & \text{otherwise} \\
            \end{cases}
        \]
        For any upwards closed subset \(U \ni z\), \(\theta_i(U) = \theta_1(z) + \theta_i'(U)\), and
        \[
            \theta_1(z) + \theta_1'(U)
            = \theta_1(U)
            \le \theta_2(U)
            = \theta_1(z) + \theta_2'(U)
        \]
        By cancellativity, \(\theta_1'(U) \le \theta_2'(U)\).
        For any upwards closed subset \(U\) such that \(z \notin U\), \(\theta_i(U) = \theta_i'(U)\), and again
        \(
            \theta_1'(U) \le \theta_2'(U)
        \).
        It follows that \(\theta_1' \sqsubseteq \theta_2'\), so from the induction hypothesis  we have \(f^\beta(\theta_1') \le f^\beta(\theta_2')\).
        Since \(f^\beta\) is a homomorphism,  
        \begin{align*}
            f^\beta(\theta_1) 
            &= f^\beta(\theta_1(z)\delta_z + \theta_1')  \\
            &=  f^\beta(\theta_1(z)\delta_z) + f^\beta(\theta_1')  \\
            &\le  f^\beta(\theta_1(z)\delta_z) + f^\beta(\theta_2') \\
            &= f^\beta(\theta_1(z) \delta_z + \theta_2') \\
            &= f^\beta(\theta_2)
        \end{align*}
        Therefore \(f^\beta\) is monotone when \(\mathbf P\) is cancellative and difference-ordered.
        
        Now assume \(\mathbf P\) is idempotent.      
        Let \(z\) be any maximal element of \(\supp(\theta_1)\) and define
        \[
            \theta_1'(x) = \begin{cases}
                0 & x = z \\
                \theta_1(x) & x \neq z
            \end{cases}
        \]
        Then \(\theta_1' \sqsubseteq \theta_2\), since \(\theta_1(U) = \theta_1'(U) + \theta_1(z)\) if \(z \in U\) and \(\theta_1(U) = \theta_1'(U)\) if \(z \notin U\).
        By the induction hypothesis, \(f^\beta(\theta_1') \le f^\beta(\theta_2)\).
        
        On the other hand, \(\theta_1(z)\delta_z \sqsubseteq \theta_2\) by assumption. 
        This means there is some \(y \in \supp(\theta_2)\) such that \(z \le y\) and \(\theta_1(z) \le \theta_2(y)\).
        We therefore have 
        \begin{align*}
            f^\beta(\theta_1) 
            &= f^\beta(\theta_1(z)\delta_z + \theta_1')  \\
            &= \theta_1(z) f(z) + f^\beta(\theta_1') \\
            &\le \theta_1(z)f(z) + f^\beta(\theta_2) \\
            &\le \theta_2(y)f(y) + f^\beta(\theta_2) \\
            &\le f^\beta(\theta_2) + f^\beta(\theta_2) \\
            &= f^\beta(\theta_2) 
        \end{align*}
        this shows that \(f^\beta\) is monotone when \(\mathbf P\) is idempotent.
    \end{proof}

    In fact, if \(\mathbf P\) is cancellative, like in \cref{eg:reals}, then \(\sqsubseteq\) is a partial order.
    Consequently, by \cref{thm:main}, if \(\mathbf P\) is cancellative and difference ordered, \((\bar\O_{\mathbf P}, [\delta], [\rho])\) lifts the monad \((\O_{\mathbf P}, \delta, \rho)\) on \(\Sets\). 

    \begin{restatable}{theorem}{cancellativedoesthetrick}
        If \(\mathbf P\) is cancellative, then \((\bar\O, [\delta], [\rho])\) lifts \((\O, \delta, \rho)\).
    \end{restatable}

    \begin{proof}
        We show that \(\sqsubseteq\) is antisymmetric.
        Let \(\theta_1 \sqsubseteq \theta_2 \sqsubseteq \theta_1\). 
    %
    %
        Define \(\theta_i'(x) = \theta_i(x)\) when \(x \neq x_0\) and \(\theta_i'(x_0) = \theta_i(x_0) - \theta_1(x_0)\), \(i = 1,2\).
        Since \(\theta_i = \theta_1(x_0)\cdot \delta_{x_0} + \theta_i'\), by cancellativity of \(\mathbf P\) 
        \[
            \theta_1(x_0)\cdot \delta_{x_0} + \theta_1' 
            \sqsubseteq \theta_1(x_0)\cdot \delta_{x_0} + \theta_2'
            \sqsubseteq \theta_1(x_0)\cdot \delta_{x_0} + \theta_1'
        \]
        implies \(\theta_1' \sqsubseteq \theta_2' \sqsubseteq \theta_1'\). 
        By induction, \(\theta_1' = \theta_2'\), so \(\theta_1 = \theta_2\) as well.
    \end{proof}
    
    \begin{remark}
        The assumption that \(\mathbf P\) is cancellative is necessary here. 
        Consider the example of \(\mathbf P = \mathbf 2 = (\{0,1\}, \le, 0, 1, \max, \min)\).
        The heavier-higher order is not a partial order on \(\mathcal O_\omega \{0,1\}\), where \(0 < 1\), because it cannot distinguish between \(\delta_{1}\) and \(\delta_0 + \delta_1\).
    \end{remark}

    \cref{thm:main} also extends to products of cancellative difference-ordered and idempotent semirings.

    \begin{corollary}
        Suppose \(\mathbf P = \prod_{i \in I} \mathbf P_i\), where \(\mathbf P_i\) is either cancellative and difference ordered, or idempotent, for each \(i \in I\).
        Then \({(\bar\O_{\mathbf P}, [\delta], [\rho])}\) is the \(\PMod{\mathbf P}\)-presented monad. 
    \end{corollary}

    Despite the apparent necessity of the extra assumptions in \cref{thm:main}, the construction \((\bar\O_{\mathbf P}, [\rho])\) always produces a \(\mathbf P\)-module, and the transformation \([\delta] : \Id \Rightarrow \bar\O_{\mathbf P}\) is always monotone with respect to the heavier-higher order. 
    One might wonder if this construction is therefore more general.

    \begin{question}
        Are the restrictions on \(\mathbf P\) in \cref{thm:main} necessary? 
        In other words, is there a semiring \(\mathbf P\) such that \(\bar\O_{\mathbf P}\) does \emph{not} construct the free ordered \(\mathbf P\)-module?
    \end{question}

    A negative answer to this question would consist of an example of an ordered semiring \(\mathbf P\) (that cannot be decomposed into cancellative difference ordered and idempotent semirings) and a monotone map \(f : (X, \le) \to (Y, \le)\) into a \(\mathbf P\)-module \((Y, \le, \beta)\) such that the inductively defined homomorphism \(f^\beta : \O_{\mathbf P} X \to Y\) is not monotone with respect to the heavier-higher order. 

    \begin{example} 
        Let \(\mathbf A\) be the ordered semiring \((A, \le, \vec 0, \vec 1, +, \cdot)\), where 
        \[
            A = \left\{f : \mathbb N \to \mathbb N ~\middle|~ \text{either \(f\) is constant or}~(\forall n)~f^{-1}(n)~\text{is finite}\right\}  
        \]
        and \(\le\) is the point-wise order. 
        Then \(\mathbf A\) is cancellative, so \(\sqsubseteq\) is a partial order. 
        On the other hand, \(\mathbf A\) is not difference-ordered: consider \(f_1(n) = n\) and \(f_2(n) = n + (n \mod 2)\). 
        Then \(f_1 \le f_2\), but \((f_2 - f_1)(n) = n \mod 2\), which is \(0\) on an infinite proper subset of \(\mathbb N\).
        This means that \(\mathbf A\) cannot be of the form described in \cref{thm:main}, as a product of difference-ordered semirings is always difference-ordered. 
        Is this a counterexample?
    \end{example}

    \section{Distribution theories} 
    Finally, I would like to consider an application of our study of ordered semirings to subprobabilistic state-based systems.
    For a given ordered semiring \(\mathbf P\), define
    \[
        \D_{\mathbf P} X = \left\{\theta \in \O_{\mathbf P}X ~\middle|~ \sum_{x \in X} \theta(x) \le 1\right\}
        \qquad\qquad 
        \bar\D_{\mathbf P} (X, \le) = (\D_{\mathbf P}X/{\equiv}, \sqsubseteq)   
    \] 
    With the added \(\bar\D_{\bf P}(h) = h^\bullet\), \(\bar\D_{\bf P}\) is a subfunctor of \(\bar\O_{\bf P}\). 
    Standard subprobability distributions are obtained by setting \(\mathbf P = \mathbf R^+\), and we will simply write \(\D\) instead of \(\D_{\mathbf R^+}\). 

    I define a (\emph{subprobabilistic}) \emph{state-based decision process} (or \emph{SDP}) to be a function of the form \(\gamma: X \to \D Y\), where \(X\) is the \emph{state space} and \(Y\) is a set, disjoint from \(X\), that represents extra data like state transitions or outputs of the machine.
    In~\cite{schmid2022ordered,schmidrozowskisilvarot2022processes} for example, \(Y = \Var + \Act\times X\) for fixed sets \(\Var\) and \(\Act\).

    For a concrete example of a subprobabilistic SDP, consider rolling a six-sided die. 
    There is one state in this system, with six outputs. 
    
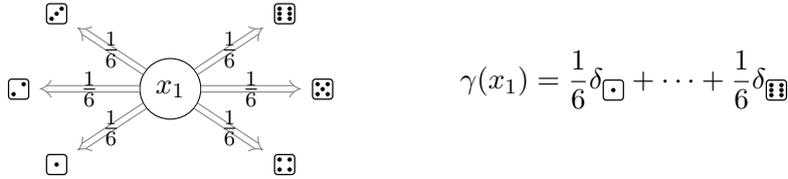
\[\begin{tikzpicture}[baseline=(current bounding box.center)]
        \node[state] (0) at (0,0) {\(x_1\)};
        \node (1) at (-1.5, -1) {\(\epsdice 1\)};
        \node (2) at (-2, -0) {\(\epsdice 2\)};
        \node (3) at (-1.5, 1) {\(\epsdice 3\)};
        \node (4) at (1.5, -1) {\(\epsdice 4\)};
        \node (5) at (2, 0) {\(\epsdice 5\)};
        \node (6) at (1.5, 1) {\(\epsdice 6\)};
        \draw (0) edge[-implies, double, double distance=2pt, color=gray] node[color=black] {\(\frac16\)} (1);
        \draw (0) edge[-implies, double, double distance=2pt, color=gray] node[color=black] {\(\frac16\)} (2);
        \draw (0) edge[-implies, double, double distance=2pt, color=gray] node[color=black] {\(\frac16\)} (3);
        \draw (0) edge[-implies, double, double distance=2pt, color=gray] node[color=black] {\(\frac16\)} (4);
        \draw (0) edge[-implies, double, double distance=2pt, color=gray] node[color=black] {\(\frac16\)} (5);
        \draw (0) edge[-implies, double, double distance=2pt, color=gray] node[color=black] {\(\frac16\)} (6);
    \end{tikzpicture}
    \qquad\qquad
    \gamma(x_1) = \frac16 \delta_{\epsdice 1} + \cdots + \frac16\delta_{\epsdice 6} \]
    Unfortunately, implementing this decision process in everyday life requires you to have a six-sided die on hand. 
    As Knuth and Yao observe in~\cite{knuthyao1976random}, however, a series of coin flips can do the same job.
    The Knuth-Yao algorithm can be visualised as a sort of state-based system, except that it involves state-to-state transitions that carry probabilities.
    
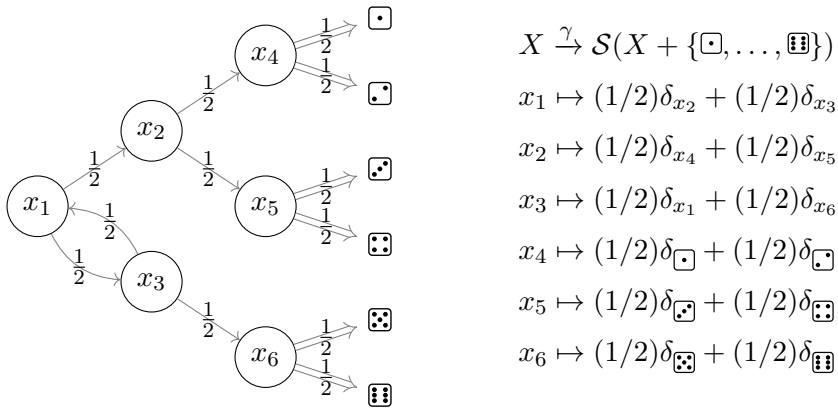
\[\begin{tikzpicture}[baseline=(current bounding box.center)]
        \node[state] (1) at (0,0) {\(x_1\)};
        \node[state] (2) at (1.5,1) {\(x_2\)};
        \node[state] (3) at (1.5,-1) {\(x_3\)};
        \node[state] (4) at (3, 2) {\(x_4\)};
        \node[state] (5) at (3, 0) {\(x_5\)};
        \node[state] (6) at (3, -2) {\(x_6\)};
        \node (7) at (4.5, 2.5) {\({\large\epsdice 1}\)};
        \node (8) at (4.5, 1.5) {\({\large\epsdice 2}\)};
        \node (9) at (4.5, 0.5) {\({\large\epsdice 3}\)};
        \node (10) at (4.5, -0.5) {\({\large\epsdice 4}\)};
        \node (11) at (4.5, -1.5) {\({\large\epsdice 5}\)};
        \node (12) at (4.5, -2.5) {\({\large\epsdice 6}\)};

        \draw[color=gray] (1) edge[->] node[color=black] {\(\frac12\)} (2);
        \draw[color=gray] (1) edge[->, bend right] node[color=black] {\(\frac12\)} (3);
        \draw[color=gray] (3) edge[->, bend right] node[color=black] {\(\frac12\)} (1);
        \draw[color=gray] (2) edge[->] node[color=black] {\(\frac12\)} (4);
        \draw[color=gray] (2) edge[->] node[color=black] {\(\frac12\)} (5);
        \draw[color=gray] (3) edge[->] node[color=black] {\(\frac12\)} (6);
        \draw[color=gray] (4) edge[-implies, double, double distance=2pt] node[color=black] {\(\frac12\)} (7);
        \draw[color=gray] (4) edge[-implies, double, double distance=2pt] node[color=black] {\(\frac12\)}  (8);
        \draw[color=gray] (5) edge[-implies, double, double distance=2pt] node[color=black] {\(\frac12\)} (9);
        \draw[color=gray] (5) edge[-implies, double, double distance=2pt] node[color=black] {\(\frac12\)}  (10);
        \draw[color=gray] (6) edge[-implies, double, double distance=2pt] node[color=black] {\(\frac12\)} (11);
        \draw[color=gray] (6) edge[-implies, double, double distance=2pt] node[color=black] {\(\frac12\)}  (12);
    \end{tikzpicture}
    \qquad\qquad 
    \begin{aligned}
        \omit\rlap{\(X \xrightarrow{\gamma} \D(X + \{\epsdice1,\dots,\epsdice6\})\)}
        \\
        x_1 &\mapsto (1/2) \delta_{x_2} + (1/2) \delta_{x_3} \\
        x_2 &\mapsto (1/2) \delta_{x_4} + (1/2) \delta_{x_5} \\
        x_3 &\mapsto (1/2) \delta_{x_1} + (1/2) \delta_{x_6} \\
        x_4 &\mapsto (1/2) \delta_{\epsdice1} + (1/2) \delta_{\epsdice2} \\
        x_5 &\mapsto (1/2) \delta_{\epsdice3} + (1/2) \delta_{\epsdice4} \\
        x_6 &\mapsto (1/2) \delta_{\epsdice5} + (1/2) \delta_{\epsdice6}
    \end{aligned}
    \]
    Recognizing the six-sided die hidden in the SDP above amounts to determining the probability of \emph{eventually} outputting one of \(\epsdice1,\dots,\epsdice6\) if the initial state is \(x_1\).
    This produces the subprobabilistic state-based system \(\gamma^\dagger : X \to \D Y\) below, 
    \begin{center}
    \begin{tabular}{c | c c c c c c}
        \(\gamma^\dagger(x_i)(-)\) & \(\epsdice1\) & \(\epsdice2\) &\(\epsdice3\) &\(\epsdice4\) &\(\epsdice5\) &\(\epsdice6\) \\
        \hline
        \(x_1\) & \(1/6\)& \(1/6\)& \(1/6\)& \(1/6\)& \(1/6\)& \(1/6\) \\
        \(x_2\) & \(1/4\)& \(1/4\)& \(1/4\)& \(1/4\)& \(0\)& \(0\) \\
        \(x_3\) & \(1/12\)& \(1/12\)& \(1/12\)& \(1/12\)& \(1/3\)& \(1/3\) \\
        \(x_4\) & \(1/2\)& \(1/2\)& \(0\)& \(0\)& \(0\)& \(0\) \\
        \(x_5\) & \(0\)& \(0\)& \(1/2\)& \(1/2\)& \(0\)& \(0\) \\
        \(x_6\) & \(0\)& \(0\)& \(0\)& \(0\)& \(1/2\)& \(1/2\) \\
    \end{tabular}
    \end{center}
    Notice that \(x_1\) represents the six-sided die system from before. 
    Algebraically, the defining feature of \(\gamma^\dagger\) is that it is the unique solution to the recursive specification
    \begin{align*}
        x_1 &= (1/2) {x_2} + (1/2) {x_3} \\
        x_2 &= (1/2) {x_4} + (1/2) {x_5} \\
        x_3 &= (1/2) {x_1} + (1/2) {x_6} \\
        x_4 &= (1/2) {\epsdice1} + (1/2) {\epsdice2} \\
        x_5 &= (1/2) {\epsdice3} + (1/2) {\epsdice4} \\
        x_6 &= (1/2) {\epsdice5} + (1/2) {\epsdice6} 
    \end{align*}
    in the sense that if you replaced each \(x_i\) with \(\gamma^\dagger(x_i)\), you would see six true identities. 
    
    Let us call an SDP \(X \to \D_{\mathbf P}Y\) \emph{unguarded} when \(X \cap Y \neq \emptyset\), and note that every unguarded system can be written in the form \(\gamma : X \to \D_{\mathbf P}(X + Y)\) for some disjoint \(X\) and \(Y\).
    In the case of sets, it is clear which unguarded SDPs have unique solutions and which do not: \(x_1 = x_1\) has many solutions, for example, but \(x_1 = (1/3)x_1 + (2/3)y\) has only one. 
    Comparing subdistributions point-wise, one can show that every recursive specification has a \emph{least} solution. 
    What I would like to investigate next is for which ordered semirings \(\mathbf P\) we can guarantee that every finite\footnote{In the sense that \(X\) is finite.} unguarded SDP with ordered outputs \(\gamma : (X, =) \to \bar\D_{\mathbf P}((X, =) + (Y, \le))\) has a least solution \(\gamma^\dagger : (X, =) \to \bar\D_{\mathbf P}(Y, \le)\).
    
    One way to ensure that every unguarded SDP has a least solution is to mimic the conditions of \(\mathbf R^+\).
    To this end, call an ordered semiring \(\mathbf P\) \emph{division} if \((\forall r > 0)(\exists s)~rs = sr = 1\).
    Note that \(s\) is necessarily unique, so we write \(s = r^{-1}\). 
    
    Given a poset \((Y, \le)\), the set \(\D_{\mathbf P}(Y, \le)\) can always be equipped with the \emph{point-wise} order, which simply has \(\theta_1 \le \theta_2\) if and only if \((\forall x)~\theta_1(x) \le \theta_2(x)\).
    Of course, if \(\theta_1 \le \theta_2\), then \(\theta_1 \sqsubseteq \theta_2\).

    \begin{restatable}{theorem}{leastfixedpoints}\label{thm:least fixed-points}
        If \(\mathbf P\) is cancellative, difference ordered, and division, then every finite subprobabilististic unguarded system \(\gamma : (X, =) \to \D_{\mathbf P}((X, =) + (Y, \le))\) has a least solution \(\gamma^\dagger : (X, =) \to \D_{\mathbf P} (Y, \le)\) with respect to the point-wise order. 
        Furthermore, if \(\gamma(x)\) is a probability\footnote{Meaning \(\gamma(x)(X) = 1\).} distribution for each \(x \in X\), then \(\gamma^\dagger(x)\) is as well. 
    \end{restatable}
    
    \begin{proof}[Proof of \cref{thm:least fixed-points}.]
        Let \(\gamma : (X, =) \to \D_{\mathbf P}((X, =) + (Y, \le))\), \(X = \{x_1, \dots, x_n\}\), and for any \(i,j \le n\) write \(a_{ij} = \gamma(x_i)(x_j)\) and \(\theta_i = \gamma(x_i)\delta_Y\).
        We need to find a least solution to the system of equations
        \begin{align*}
            x_1 &= a_{11}x_1 + \cdots + a_{1n}x_n + \theta_1\\
                &\ \ \vdots \tag{1}\\
            x_n &= a_{n1}x_1 + \cdots + a_{nn}x_n + \theta_n
        \end{align*}
        We proceed by induction on \(n\).
        If \(n = 1\), then the system consists of the single equation \(x_1 = a_{11}x_1 + \theta_1\). 
        If \(a_{11} = 0\), then \(\gamma^\dagger(x_1) = \theta_1\) is the least solution.
        If \(a_{11} = 1\), then again, \(\gamma^\dagger(x_1) = 0\) is the least solution, as \(\theta_1 = 0\).
        Otherwise, \(0 < a_{11} < 1\) and \(1 - a_{11} > 0\). 
        Let \(b \in \mathbf P\) such that \((1 - a_{11})b = 1\), and observe that \(\gamma^\dagger(x_1) = b\theta_1\) is a solution:
        \begin{align*}
            a_{11}b\theta_1 + \theta_1
            = (a_{11}b + 1)\theta_1
            = (a_{11}b + (1-a_{11})b)\theta_1
            = (a_{11}+ (1-a_{11}))b\theta_1
            = b\theta_1
        \end{align*} 
        It is also easy to see that it is the unique solution: if \(\psi\) is another, then \(\psi(x_i) = a_{11}\psi(x_1) + \theta_1\) and therefore \((1 - a_{11})\psi(x_1) = \psi(x_1) - a_{11}\psi(x_1) = \theta_1\). 
        Multiplying both sides by \(b\), 
        \[
            \psi(x_1) = b(1 - a_{11})\psi(x_1) = b\theta_1 = \gamma^\dagger(x_1)
        \]
    
        For the inductive step, consider the \(n\)th equation in (1).
        Without loss of generalisation, assume that \(a_{nn} > 0\).
        If \(a_{nn} = 1\), then \(a_{ni} = 0\) for \(i \neq n\) and \(\theta_n = 0\), and we let \(\gamma^\dagger(x_n) = 0\).
        Substituting \(x_n\) for \(\gamma^\dagger(x_n)\), we see that (1) is equivalent to 
        \begin{align*}
            x_1 &= a_{11}x_1 + \cdots + a_{1(n-1)}x_{n-1} + \theta_1\\
                &\ \ \vdots \\
            x_{n-1} &= a_{(n-1)1}x_1 + \cdots + a_{(n-1)(n-1)}x_{n-1} + \theta_{n-1}
        \end{align*}
        which, by the induction hypothesis has a least solution \(\psi\).
        Setting \(\gamma^\dagger(x_i) = \psi(x_i)\) for \(i < n\) and \(\gamma^\dagger(x_n) = 0\), then \(\gamma^\dagger\) is a solution to (1).  
        If \(\phi\) is another solution to (1), then \(\gamma^\dagger(x_i) \le \phi(x_i)\) for \(i < n\) by the induction hypothesis, and \(\gamma^\dagger(x_n) \le \phi(x_n)\) trivially.  
    
        If \(0 < a_{nn} < 1\), let \(b = (1 - a_{nn})^{-1}\).
        Substituting \(x_n\) for \(ba_{n1}x_1 + \cdots + ba_{n(n-1)}x_{n-1} + b\theta_n\) in (1), expanding, and collecting like terms, we obtain
        \begin{align*}
            x_1 &= c_{11}x_1 + \cdots + c_{1(n-1)}x_{n-1} + (\theta_1 + b\theta_n)\\
                &\ \ \vdots \tag{2}\\
            x_{n-1} 
                &= c_{(n-1)1}x_1 + \cdots + c_{(n-1)(n-1)}x_{n-1} + (\theta_{n-1} + a_{(n-1)n}b\theta_n)
        \end{align*}
        where \(c_{ij} = a_{ij} + a_{in}ba_{nj}\).
        By the induction hypothesis, (2) has a least solution \(\psi\). 
        To obtain a solution to (1), set \(\gamma^\dagger(x_i) = \psi(x_i)\) for \(i < n\) and
        \[
           \gamma^\dagger(x_n) = ba_{n1}\psi(x_1) + \cdots + ba_{n(n-1)}\psi_{n-1} + b\theta_n
        \]
        To see that \(\gamma^\dagger\) is indeed a solution, observe that for \(i < n\), 
        \begin{align*}
            \gamma^\dagger(x_i) 
            &= c_{i1}\gamma^\dagger(x_1) + \cdots + c_{i(n-1)}\gamma^\dagger(x_{n-1}) + (\theta_i + b\theta_n) \\
            &= (a_{i1} + a_{in}ba_{n1})\gamma^\dagger(x_1) + \cdots + (a_{i(n-1)} + a_{in}ba_{n{n-1}})\gamma^\dagger(x_{n-1}) + (\theta_i + b\theta_n) \\
            &= a_{i1}\gamma^\dagger(x_1) + \cdots + a_{i(n-1)}\gamma^\dagger(x_{n-1}) + a_{in}\gamma^\dagger(x_n) + \theta_i 
        \end{align*}
        as well as 
        \begin{align*}
            (1 - a_{nn})\gamma^\dagger(x_n) 
            &= (1 - a_{nn})ba_{n1}\gamma^\dagger(x_1) + \cdots + (1 - a_{nn})ba_{n(n-1)}\gamma^\dagger(x_{n-1}) + (1 - a_{nn})b\theta_n \\
            &= a_{n1}\gamma^\dagger(x_1) + \cdots + a_{n(n-1)}\gamma^\dagger(x_{n-1}) + \theta_n \\
            \gamma^\dagger(x_n)
            &= a_{n1}\gamma^\dagger(x_1) + \cdots + a_{n(n-1)}\gamma^\dagger(x_{n-1}) + a_{nn}\gamma^\dagger(x_n) + \theta_n 
        \end{align*}
        To see that it is the least solution, let \(\phi\) be any solution to (1), and observe that this implies
        \[
            \phi(x_n) = ba_{n1}\phi(x_1) + \cdots + ba_{n(n-1)}\phi(x_{n-1}) + b\theta_1
        \]
        and for \(i < n\),
        \[
            \phi(x_i)
            = c_{i1}\phi(x_1) + \cdots + c_{i(n-1)}\phi(x_{n-1}) + (\theta_i + b\theta_n)
        \]
        This means \(\phi\) is a solution to (2), so \(\gamma^\dagger(x_i) = \psi(x_i) \le \phi(x_i)\) for each \(i < n\) by the induction hypothesis. 
        Finally,
        \begin{align*}
            \gamma^\dagger(x_n) 
            &= ba_{n1}\gamma^\dagger(x_1) + \cdots + ba_{n(n-1)}\gamma^\dagger(x_{n-1}) + b\theta_n\\
            &= ba_{n1}\psi(x_1) + \cdots + ba_{n(n-1)}\psi(x_{n-1}) + b\theta_n\\
            &\le ba_{n1}\phi(x_1) + \cdots + ba_{n(n-1)}\phi(x_{n-1}) + b\theta_n \\
            &= \phi(x_n) 
        \end{align*}
        This shows that \(\gamma^\dagger\) is the least solution to (1) in the point-wise order.
    \end{proof}

    \begin{example}
        Consider the ordered semiring \(\mathbf R^n_> = (R, <^*, \vec 0, \vec 1, +, \cdot)\), where 
        \[
            R = \{(r_1, \dots, r_n) \in (\mathbb R^+)^n \mid r_i = 0 \implies (\forall j)~r_j = 0\}
        \]
        where \(\vec r <^* \vec s\) iff either \(\vec r = \vec s\) or \(r_i < s_i\) for all \(i \le n\).
        Then \(\mathbf R^n_>\) is division, and furthermore satisfies the conditions of \cref{thm:least fixed-points}. 
        An unguarded system \((X, =) \to \D_{\mathbf R^n_>}((X, =), (Y, \le))\) can be thought of as \(n\) simultaneous unguarded systems \((X, =) \to \D_{\mathbf R^+}((X, =)+(Y, \le))\).
    \end{example}

    Buried in the proof of \cref{thm:least fixed-points} is a proof of the following additional fact: if \(0 < \gamma(x_i)(x_j) < 1\) for all \(x_i,x_j \in X\), then \(\gamma\) has a \emph{unique} solution. 
    This somewhat trivially implies that if \(\mathbf P\) is cancellative, difference ordered, and division, then every finite subprobabilististic unguarded SDP has a least solution with respect to the heavier-higher order as well.
    
    A different approach shows that divisibility is not necessary if we assume that \(\mathbf P\) is sufficiently complete.
    An ordered algebra is called \emph{\(\omega\)-continuous} if every chain \(x_1 \le x_2 \le \dots\) has a least upper bound and its algebraic operations preserve least upper bounds of chains. 

    \begin{theorem}
        Let \(\mathbf P\) be either cancellative and difference-ordered or idempotent. 
        If \(\mathbf P\) is \(\omega\)-continuous, then every unguarded system has a least solution.
    \end{theorem}

    \begin{proof}
        This is nearly an application of the Kleene fixed-point theorem~\cite{kleene1952fpt}.
        If we let 
        \[
            L : {\bar\D_{\mathbf P}((X, =) + (Y, \le))^n} \to {\bar\D_{\mathbf P}((X, =) + (Y, \le))^n}
            \qquad
            L\begin{bmatrix}
                x_1 \\ \vdots \\x_n
            \end{bmatrix}  
            =\begin{bmatrix}
            a_{11}x_1 + \cdots + a_{1n}x_n + \theta_1\\
            \vdots\\
            a_{n1}x_1 + \cdots + a_{nn}x_n + \theta_n
            \end{bmatrix}
        \]
        as in (1), and \(\mathbf P\) be \(\omega\)-continuous, then by monotonicity and point-wise \(\omega\)-continuity of \(L\), we just need to take the least upper bound of \(L^n(\vec 0)\) to find the least solution of (1).
        There is a small snag in this approach: even if \(\mathbf P\) is continuous, \(\bar\D_{\mathbf P}(X, \le)\) is usually not.
        Consider, for example, \(X = \mathbb N\), and the chain defined by \(\psi_n = \delta_n + \sum_{m < n} \delta_m\). 
        Any upper bound to \(\{\psi_n \mid n \in \mathbb N\}\) must have infinite support, and so will lie outside of \(\bar\D_{\mathbf P}(\mathbb N, \le)\).
        
        On the other hand, the support of the \(i\)th component of \(L^n(\vec 0)\) is contained in the union \(U \subseteq Y\) of the supports of the \(\theta_j\).
        Since there are finitely many \(\theta_j\) to consider, \(U\) is finite.
        Now, \(L\) is monotone, so there is an \(N\in \N\) such that the support of the \(i\)th component of \(L^m(\vec 0)\) is equal to the support \(U_i\) of the \(i\)th component of \(L^N(\vec 0)\) for all \(n > N\).

        For each \(y \in U_i\), let \(\psi_i(y) = \sup (L^n(\vec 0))_i(y)\) and \(\psi_i(x) = 0\) for each \(x \notin U_i\). 
        Then \(\psi_i\) is a subprobability distribution because \(L^n(\vec 0)_i\) is for each \(n \in \mathbb N\). 
        Also, a routine check reveals that \(\vec \psi\) is the least upper bound of \(L^n(\vec 0)\).
        To see that \(\vec \psi\) is the least fixed-point, use continuity to show that \(L(\vec \psi) \sqsubseteq \vec \psi\).
        For any other \(\vec\gamma\) such that \(L(\vec \gamma) \sqsubseteq \vec \gamma\), \(L^n(\vec 0)_i(y) \le \gamma_i(y)\) for all \(n > N\) and \(y \in U_i\).
        Therefore \(\vec \psi \sqsubseteq \vec \gamma\), so indeed \(\vec\psi\) is a least fixed-point of \(L\). 
    \end{proof}

	\printbibliography

\end{document}